\documentclass[10pt]{article}
\usepackage{fullpage}
\usepackage{graphicx,amsmath,amsthm,epstopdf,tikz,enumitem,color,%
	float,amsfonts,wasysym,etoolbox}

\usetikzlibrary{backgrounds,arrows.meta}

\usepackage[mathscr]{euscript}
\usepackage[colorlinks=true,allcolors=blue]{hyperref}
\usepackage[capitalise,noabbrev,nameinlink]{cleveref}

\usepackage{caption}
\usepackage[subrefformat=simple,labelformat=simple]{subcaption}

\numberwithin{figure}{section}
 
\def\ignoretwo#1#2{}
\ifdef\IfFormatAtLeastTF{%
  \IfFormatAtLeastTF{2024-11-01}{% new LaTeX
    \immediate\write16{New LaTeX, adding hooks for cleveref.}%
    \let\CrefAddToHook=\AddToHook%
  }{% old LaTeX 2020-2024
    \immediate\write16{Old LaTeX, not adding hooks for cleveref.}%
    \let\CrefAddToHook=\ignoretwo%
  }%
}{% very old LaTeX before 2020, \IfFormatAtLeastTF undefined
  \immediate\write16{Very old LaTeX, not adding hooks for cleveref.}%
  \let\CrefAddToHook=\ignoretwo%
}%

\usepackage{thm-restate}

\newtheorem{theorem}{Theorem}[section]
 \CrefAddToHook{env/theorem/begin}{\crefalias{section}{theorem}}
\newtheorem{lemma}[theorem]{Lemma}
 \CrefAddToHook{env/lemma/begin}{\crefalias{section}{lemma}}
\newtheorem{corollary}[theorem]{Corollary}
 \CrefAddToHook{env/corollary/begin}{\crefalias{section}{corollary}}

\newtheorem{observation}[theorem]{Observation}
 \CrefAddToHook{env/observation/begin}{\crefalias{section}{observation}}
 \crefname{observation}{Observation}{Observations}
\crefname{appendix}{Appendix}{Appendices}
\crefname{subsection}{Subsection}{Subsections}
\makeatletter
\CrefAddToHook{cmd/appendix/before}{%
  \crefalias{section}{appendix}%
  \crefalias{subsection}{appendix}%
}
\makeatother

\newlist{statement}{enumerate}{1}
\setlist[statement]{label=\textup{(\alph*)},ref=\doublelabel{\textup{(\alph*)}},before=\setcrefdoublealias,nosep}
\crefname{statementi}{statement}{statements} %% arguments are: name of counter, singular name, plural name

\makeatletter
\newcommand*\doublelabel[1]{\protect\@twolabels{#1}{\@currentlabel#1}}
\let\@twolabels\@firstoftwo

\def\setcrefdoublealias{%
	\begingroup\edef\x{\endgroup%
		\noexpand\crefalias{statementi}{%
			\noexpand\protect\noexpand\@twolabels%
			{statementi}{\expandafter\@extractcounterfromcreflabel\cref@currentlabel\end@extractcounterfromcreflabel}%
		}%
	}\x%
}
\def\@extractcounterfromcreflabel[#1]#2\end@extractcounterfromcreflabel{#1}

\newcommand*\versionWithTheorem[1]{\@ifstar{\versionWithTheorem@aux{#1*}}{\versionWithTheorem@aux{#1}}}
\newcommand*\versionWithTheorem@aux[2]{\begingroup\let\@twolabels\@secondoftwo#1{#2}\endgroup}
\makeatother

\DeclareRobustCommand*\crefWithTheorem{\versionWithTheorem\cref}

\newcommand{\Loneinf}{\mathscr{L}_{\infty}}
\newcommand{\Ltwoinf}{\mathscr{L}^{\Delta}_{\infty}}
\newcommand{\Lthreeinf}{\mathscr{L}^{\nabla\Delta}_{\infty}}
\def\tri{^\Delta}
\def\tritri{^{\nabla\Delta}}
\def\fref#1{f_{\ref{#1}}}

\long\def\ignore#1{}

\title{Unavoidable substructures in large and infinite $2$-edge-connected graphs}
\author{Sarah Allred\thanks{Department of Mathematics \& Statistics, University of South Alabama, Mobile, AL, USA (sarahallred@southalabama.edu)}
\and M.~N.~Ellingham\thanks{Department of Mathematics,  Vanderbilt University, Nashville, TN, USA (mark.ellingham@vanderbilt.edu)}
}

\date{June 8, 2026}% Use specific date for submitted version
\begin{document}
\maketitle

\begin{abstract}
In 1930, Ramsey proved that every large graph contains either a large clique or a large edgeless graph as an induced subgraph.  It is well known that every large connected graph contains a long path, a large clique, or a large star as an induced subgraph.  Recently Allred, Ding, and Oporowski presented the unavoidable large induced subgraphs for large and infinite $2$-connected graphs.  The $2$-edge-connected (sometimes called bridgeless) graphs form an important class between connected graphs and $2$-connected graphs.  In this paper we describe the unavoidable large induced subgraphs for large and infinite $2$-edge-connected graphs.
Ubiquitous structures in $2$-edge-connected graphs that we call \emph{chains of pinched super-clean ladders} play an important role in these descriptions.
As consequences we obtain results on unavoidable large subgraphs, topological minors, minors, induced topological minors, induced minors, and Eulerian subgraphs in large and infinite $2$-edge-connected graphs.  When appropriate we extend our results to multigraphs.
\end{abstract}

\begingroup\small
\endgroup

%AMS Subject Class
%\begin{AMS}
%		Primary 05C75; Secondary 05C55
%	\end{AMS}

\section{Introduction}

\subsection{Overview}

In this paper graphs are simple, and multigraphs may have multiple edges but not loops.  Graphs and multigraphs are assumed to be finite unless we explicitly indicate that they are infinite.  Terms and symbols not defined here follow West~\cite{west}.  For notational simplicity we use $\infty$ to mean countable infinity, $\aleph_0$, and we generalize notation for finite graphs such as $K_r, K_{k,r}, K_{1,1,r}$ to countably infinite graphs $K_\infty, K_{k,\infty}, K_{1,1,\infty}$, and so on, if there is no ambiguity.  Sometimes there may be more than one infinite analog of a class of finite graphs (for example, there are one-way and two-way infinite paths), and we give specific definitions. 

The main part of this paper presents results on unavoidable large induced subgraphs in $2$-edge-connected graphs.  The class of $2$-edge-connected graphs (also called \emph{bridgeless} graphs) is the largest class of graphs with a simple connectivity condition stronger than just being connected.  A number of significant problems (such as the Cycle Double Cover Conjecture and Tutte's $5$-flow Conjecture) involve $2$-edge-connected graphs.

Currently results for unavoidable large induced subgraphs are known for all graphs, connected graphs, and $2$-connected graphs.  We discuss these results in more detail below.  To obtain our results for $2$-edge-connected graphs we use the previous $2$-connected results described in  \cref{ss:highconn}, along with a result showing that a certain type of induced subgraph is ubiquitous in $2$-edge-connected graphs.

Our main results allow us to give as an \hyperref[sec:otherstructures]{appendix} a comprehensive list of unavoidable substructure results for $2$-edge-connected graphs, covering unavoidable subgraphs, topological minors, minors, induced topological minors, and induced minors.  To illustrate potential applications, we show that our results imply, and strengthen, a recent result of Goddard and LaVey \cite{Goddard2024} on large Eulerian subgraphs in $2$-edge-connected graphs, which was used to obtain rainbow walks in large edge-colored $2$-edge-connected graphs.

\subsection{Background}

Many results in graph theory depend on the fact that sufficiently large (or infinite) graphs contain one of a family of unavoidable substructures. 
For example, Ding, Oporowski, Thomas, and Vertigan \cite{crossingcrit} use results on unavoidable topological minors in $3$-connected graphs from \cite{Unavoidabletopminor3conngraphs} to obtain results on $2$-crossing-critical graphs, and Ding and Marshall \cite{GuoliEmily} use \cref{thm:indconnfin} below to obtain a characterization of graphs with no large theta graph as a minor.
Results guaranteeing the existence of unavoidable substructures are therefore of fundamental importance.

Two foundational results of this kind for infinite graphs are due to K\"{o}nig in 1927 and Ramsey in 1930.  Denote the \emph{ray} or \emph{one-way infinite path} $v_1 v_2 v_3 \dots$ by $P_\infty$.
K\"{o}nig's result includes the following as an important special case.

\begin{theorem}[K\"{o}nig \cite{konig}]\label{thm:koniginf}
Every infinite connected graph has either $K_{1,\infty}$ or $P_\infty$ as a subgraph.
\end{theorem}

Ramsey proved a result for complete uniform hypergraphs; the simplest case can be stated as follows.

\begin{theorem}[Ramsey \cite{ramsey}]\label{thm:ramseyinf}
Every infinite graph contains either $K_\infty$ or its complement $\overline{K}_\infty$ as an induced subgraph.
\end{theorem}

K\"{o}nig's result and Ramsey's result differ in two ways.
First, K\"{o}nig's result guarantees a subgraph, while Ramsey's result guarantees an induced subgraph.  
Second, K\"{o}nig's result has a connectivity condition, while Ramsey's does not.  However, both \cref{thm:koniginf,thm:ramseyinf} apply to graphs whose number of vertices is any infinite cardinal, although they guarantee only countably infinite substructures.
And both results have counterparts for sufficiently large finite graphs: \cref{thm:konigfin} is an easy exercise, and \cref{thm:ramseyfin} was also proved by Ramsey.

\begin{theorem}\label{thm:konigfin}
For every positive integer $r$, there is an integer $\fref{thm:konigfin}(r)$ such that every graph on at least $\fref{thm:konigfin}(r)$ vertices has either $K_{1,r}$ or $P_r$ as a subgraph.
\end{theorem}

\begin{theorem}[Ramsey \cite{ramsey}]\label{thm:ramseyfin}
For every positive integer $r$, there is an integer $\fref{thm:ramseyfin}(r)$ such that every graph on at least $\fref{thm:ramseyfin}(r)$ vertices contains either~$K_r$ or~$\overline K_r$ as an induced subgraph.
\end{theorem}

Besides the subgraph and induced subgraph orderings, there are results involving unavoidable substructures for the topological minor and minor orderings of graphs, based on various connectivity conditions.  The four orderings can be ranked from strongest to weakest as induced subgraph, subgraph, topological minor, and minor.
A result on unavoidable substructures in a class of graphs for a given ordering generally also yields results for all weaker orderings.  Results for weaker orderings are usually easier to state because the complicated unavoidable structures for a stronger ordering have simpler unavoidable substructures under a weaker ordering.
Oporowski, Oxley, and Thomas \cite{Unavoidabletopminor3conngraphs} found unavoidable topological minors for $2$-connected, $3$-connected, and internally $4$-connected graphs, which imply results on minors.

Induced subgraphs occur in characterizations of a number of graph classes.
For example, the class of line graphs is characterized by finitely many forbidden induced subgraphs, and chordal graphs and perfect graphs are characterized by simple infinite sets of forbidden induced subgraphs.
Several important problems in graph theory involve forbidden induced subgraphs.
For example, the Matthews-Sumner Conjecture proposes that $4$-connected claw-free graphs are hamiltonian; this is true for $6$-connected graphs \cite{KV12}.
The Gy\'arf\'as-Sumner Conjecture states that graphs that do not have a particular tree $T$ as an induced subgraph are $\chi$-bounded, i.e., $\chi(G) \le f_T(\omega(G))$ for some function $f_T$; recent results on this appear in \cite{CSSS23-polybd6,LSWY23}.

The only known connectivity-based results for unavoidable induced subgraphs are for connected and $2$-connected graphs.
For connected graphs we have the following.

\begin{theorem}\label{thm:indconninf}
Every infinite connected graph contains $K_\infty$, $K_{1,\infty}$, or $P_\infty$ as an induced subgraph.
\end{theorem}

\begin{theorem}[{\cite[(5.3)]{Unavoidableminors3connbinmatroids} or \cite[Proposition 9.4.1]{diestel}}]\label{thm:indconnfin}
For every positive integer $r$, there is an integer $\fref{thm:indconnfin}(r)$
such that every connected graph on at least $\fref{thm:indconnfin}(r)$ vertices
contains $K_r$, $K_{1,r}$, or $P_r$ as an induced subgraph.
\end{theorem}

We describe the results on unavoidable induced subgraphs for $2$-connected graphs, due to the first author, Ding, and Oporowski \cite{unavoidableinducedsubgraphs,unavoidableinfinducedsubgraphs}, in \cref{ss:highconn}.

\subsection{Main results}\label{ss:mainresults}
In a particular class of graphs, we desire first that the unavoidable substructures are still in the class of graphs, and secondly, if possible the unavoidable substructures are minimal with respect to membership in the graph class.
By \cref{thm:ramseyfin}, large complete graphs will be in the list of unavoidable induced subgraphs no matter the connectivity requirement, even though they are not minimal.

In order to state our results we need some appropriate definitions and notation.

An \emph{$r$-flower}, where $r \ge 1$ or $r = \infty$, consists of $r$ edge-disjoint induced cycles that have a single common vertex; see \cref{fig:rflower}.  
(In our figures thin line segments represent single edges, while thick line segments represent paths, which may be a single edge.)
In our characterization, an $r$-flower can be regarded as an analog of a star $K_{1,r}$.  We have a second analog of stars: $\Theta_r$, where $r \ge 3$ is an integer or $r = \infty$, is the class of graphs that consist of $r$ internally disjoint paths between two specified vertices.

For a path $P$, we denote the subpath of $P$ with initial vertex $u$ and final vertex $v$ by $P[u,v]$, and we denote the subpath $P[u,v]-\{u,v\}$ by $P(u,v)$ (which is empty if $u=v$).
The subpaths $P(u,v]$ and $P[u,v)$ are defined analogously.

A \emph{pinched ladder} is a triple $(L,P,Q)$ where
(1) $L$ is a graph and $P$ and $Q$ are two paths in $L$,
(2) $V(L) = V(P)\cup V(Q)$, 
(3) $P \cap Q$ consists of two vertices $\sigma$ and $\tau$, the \emph{initial} and \emph{final} vertices of $L$, respectively, which are also the endvertices of $P$ and $Q$, and 
(4) either $P$ is a single edge $e$ and $Q=L-\{e\}$ (or vice versa), or $P$ and $Q$ are both induced paths.
If either $P$ or $Q$ is a single edge then the pinched ladder is a cycle.
The \emph{rails} of a pinched ladder are $P(\sigma,\tau)$ and $Q(\sigma,\tau)$ (one of which may be empty).
Thus, we may suppose that $P = p_0 p_1 p_2 \dots p_\ell p_{\ell+1}$ and $Q = q_0 q_1 q_2 \dots q_m q_{m+1}$ where $p_0 = q_0 = \sigma$, $p_{\ell+1} = q_{m+1} = \tau$.  Then the rails are $p_1 p_2 \dots p_{\ell}$ (which is empty if $\ell=0$) and $q_1 q_2 \dots q_m$ (which is empty if $m=0$).
The edges of $L$ that belong to neither $P$ nor $Q$ are called \emph{rungs}.  Note that no rungs are incident with $\sigma$ or $\tau$; each rung has one end on each rail.
Two rungs $p_a q_b$ and $p_cq_d$ \emph{cross} if $a<c$ and $d<b$.
We also say that $\{p_aq_b,p_cq_d\}$ is a \emph{cross} whose \emph{$P$-span} is $P[p_a,p_c]$, and whose \emph{$Q$-span} is $Q[q_d,q_b]$; the \emph{span} is the union of the $P$-span and the $Q$-span. 
A cross whose $P$-span and $Q$-span are both single edges is called \emph{trivial}.
The edges incident with $\sigma$ or $\tau$ (or both) are \emph{terminal edges}.  Every edge of $L$ is a rail edge, a rung, or a terminal edge. 

\begin{figure}[ht]
  \begin{center}
    \begin{subfigure}[b]{.19\textwidth}
    \begin{center}
    \begin{tikzpicture}
        [scale=.8,auto=left,every node/.style={circle, fill, inner sep=0pt, minimum size=2mm, outer sep=-.2pt}]
        \node[fill=black!37!white] (a) at (1,0) {};
        \begin{scope}[on background layer]
        \draw[line width=.8mm] (a) to [out=0,in=60, looseness=40] (a);
         \draw[line width=.8mm] (a) to [out=72,in=132, looseness=40] (a);
          \draw[line width=.8mm] (a) to [out=144,in=204, looseness=40] (a);
           \draw[line width=.8mm] (a) to [out=216,in=276, looseness=40] (a);
           \draw[line width=.8mm] (a) to [out=288,in=348, looseness=40] (a);
           \end{scope}
        \end{tikzpicture}
        \end{center}
        \setlength{\abovecaptionskip}{-15pt} 
    \caption{A $5$-flower}
    \label{fig:rflower}
    \end{subfigure}
    \hspace{1cm}
    \begin{subfigure}[b]{.63\textwidth}
    \begin{center}
    \begin{tikzpicture}
        [scale=.8,auto=left,every node/.style={circle, fill, inner sep=0pt, minimum size=1.5mm, outer sep=0pt},line width=.4mm, line join=bevel]
        \node[minimum size=2mm] (u1) at (0.3,3) [label=left:$\sigma$] {};
        \node[minimum size=2mm] (u2) at (11.3,3) [label=right:$\tau$] {};
        % nodes xi, yi have x value 1+0.3*i, u1 at position -1, u2 at position 33
        \node (x0) at (1,4) {};
        \node (x2) at (1.6,4) {};
        \node (x4) at (2.2,4) {};
        \node (x6) at (2.8,4) {};
        \node (x8) at (3.4,4) {};
        \node (x10) at (4,4) {};
        \node (x12) at (4.6,4) {};
        \node (x14) at (5.2,4) {};
        \node (x16) at (5.8,4) {};
        \node (x18) at (6.4,4) {};
        \node (x20) at (7,4) {};
        \node (x22) at (7.6,4) {};
        \node (x24) at (8.2,4) {};
        \node (x26) at (8.8,4) {};
        \node (x28) at (9.4,4) {};
        \node (x30) at (10,4) {};
        \node (x32) at (10.6,4) {};
        \node (y0) at (1,2) {};
        %\node (y2) at (1.6,2) {};
        \node (y4) at (2.2,2) {};
        \node (y6) at (2.8,2) {};
        \node (y9) at (3.7,2) {};
        \node (y11) at (4.3,2) {};
        \node (y13) at (4.9,2) {};
        \node (y16) at (5.8,2) {};
        \node (y18) at (6.4,2) {};
        \node (y20) at (7,2) {};
        \node (y22) at (7.6,2) {};
        \node (y24) at (8.2,2) {};
        \node (y26) at (8.8,2) {};
        \node (y28) at (9.4,2) {};
        \node (y30) at (10,2) {};
        \node (y32) at (10.6,2) {};
 
        \begin{scope}[on background layer]
        % outside cycle bold bits and some thin bits
        \draw[line width=.4mm] (y0.center) to (u1.center) to (x0.center) to (y0.center);
        \draw[line width=.4mm] (x0.center) to (x2.center);
        \draw[line width=.8mm] (x2.center) to (x4.center);
        \draw[line width=.4mm] (x4.center) to (x6.center);
        \draw[line width=.8mm] (x6.center) to (x8.center);
        \draw[line width=.4mm] (x8.center) to (x14.center);
        \draw[line width=.8mm] (x14.center) to (x16.center);
        \draw[line width=.4mm] (x16.center) to (x18.center);
        \draw[line width=.8mm] (x18.center) to (x20.center);
        \draw[line width=.4mm] (x20.center) to (x22.center);
        \draw[line width=.8mm] (x22.center) to (x24.center);
        \draw[line width=.4mm] (x24.center) to (x26.center);
        \draw[line width=.8mm] (x26.center) to (x28.center);
        \draw[line width=.4mm] (x28.center) to (x30.center);
        \draw[line width=.8mm] (x30.center) to (x32.center);
        \draw[line width=.4mm] (x32.center) to (u2.center) to (y32.center);
        \draw[line width=.8mm] (y32.center) to (y30.center);
        \draw[line width=.4mm] (y30.center) to (y28.center);
        \draw[line width=.8mm] (y28.center) to (y26.center);
        \draw[line width=.4mm] (y26.center) to (y24.center);
        \draw[line width=.8mm] (y24.center) to (y22.center);
        \draw[line width=.4mm] (y22.center) to (y20.center);
        \draw[line width=.8mm] (y20.center) to (y18.center);
        \draw[line width=.4mm] (y18.center) to (y16.center);
        \draw[line width=.8mm] (y16.center) to (y13.center);
        \draw[line width=.4mm] (y13.center) to (y9.center);
        \draw[line width=.8mm] (y9.center) to (y6.center);
        \draw[line width=.4mm] (y6.center) to (y4.center);
        \draw[line width=.8mm] (y4.center) to (y0.center);
        % inner fans and crosses: one edge at a time to avoid artefacts
        \draw[line width=.4mm, line join=bevel] (y0.center) to (x2.center);
        \draw[line width=.4mm, line join=bevel] (x4.center) to (x6.center)to (y4.center) to (y6.center) to (x4.center);
        \draw[line width=.4mm, line join=bevel] (x8.center) to (y9.center)to (x10.center) to (y11.center) to (x12.center) to (y13.center) to (x14.center);
        \draw[line width=.4mm, line join=bevel] (x16.center) to (y16.center) to (x18.center) to (y18.center) to (x16.center);
        \draw[line width=.4mm, line join=bevel] (x20.center) to (y22.center) to (x22.center) to (y20.center);
        \draw[line width=.4mm,line join=bevel] (y24.center) to (x24.center) to (y26.center) to (x26.center);
        \draw[line width=.4mm,line join=bevel] (x28.center) to (y28.center) to (x30.center) to (y30.center);
        \end{scope}

    \end{tikzpicture}
    \end{center}
    \caption{A super-clean pinched ladder}
    \label{fig:clchain}
    \end{subfigure}
    \caption{Graphs from \cref{thm:2econ}}
    \label{fig:2econgraphs}
  \end{center}
\end{figure}

A \emph{fan} is a graph consisting of a vertex $u$ called the \emph{apex}, a path $v_1 v_2 \dots v_k$ with $k \ge 2$ called the \emph{rim}, and edges joining $u$ to $v_1$, $v_k$, and an arbitrary subset of $\{v_2, v_3, \dots, v_{k-1}\}$.
An fan is \emph{trivial} if $k=2$, i.e., it is a triangle.
An \emph{embedded fan} is an induced subgraph of a pinched ladder (or later, a ladder) that is a fan, where the apex $u$ is a vertex of one rail and the rim is the subpath of the second rail between the first and last neighbors of $u$ on that rail.  (Thus, an embedded fan is maximal: it cannot be extended to a larger embedded fan.)
We frequently abbreviate ``embedded fan'' to ``fan'' if the meaning is clear from context.

We can now describe a type of subgraph that is fundamental for $2$-edge-connected graphs.
A \emph{super-clean pinched ladder} is a pinched ladder in which all crosses and embedded fans are trivial.  See \cref{fig:clchain}.  A super-clean pinched ladder has induced subgraphs that consist of (1) a cross and any other rungs induced by its endpoints, or (2) a maximal sequence of trivial fans, where consecutive elements intersect in a rung (these zigzag between the rails).  All subgraphs of these two types are vertex-disjoint.  

A \emph{block} $B$ of a graph $G$ is a maximal connected subgraph of $G$ with no cutvertex.  Each block is a single vertex, a cutedge, or is $2$-connected, so in a $2$-edge-connected graph all blocks are $2$-connected.
The \emph{block-cutvertex tree} of a graph $G$ is a tree $T$ where each cutvertex $u$ of $G$ is a vertex of $T$, for each block $B_i$ of $G$ there is a vertex $v_i$ of $T$, and $uv_i \in E(T)$ whenever $u \in V(B_i)$.  Every leaf of $T$ corresponds to a block in $G$.

A \emph{chain of blocks} of \emph{length $n$}, or just \emph{chain of $n$ blocks}, is a graph with $n$ blocks whose block-cutvertex tree is a path.  In this case we call the cutvertices \emph{joining vertices}.  We can refer to a \emph{chain of cycles} or \emph{chain of triangles} if every block is a cycle or triangle, respectively.
A \emph{chain of super-clean pinched ladders} $H$ is a chain of blocks with blocks $L_1, L_2, \dots, L_n$ such that each $L_i$ is a super-clean pinched ladder with initial vertex $u_i$ and final vertex $u_{i+1}$.  Thus, $u_2, u_3, \dots, u_n$ are the joining vertices.  We call $u_1$ and $u_{n+1}$ the \emph{initial} and \emph{final} vertices of $H$, respectively.  See \cref{fig:mclchain}.  Large super-clean pinched ladders and long chains of super-clean pinched ladders can be thought of as analogs of a long path $P_r$.

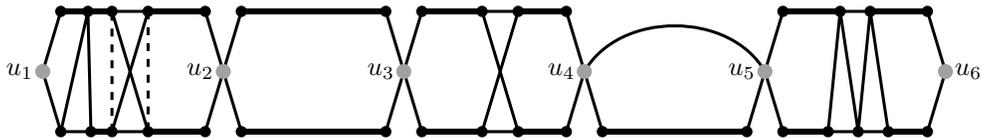
\begin{figure}[htb]
\vspace{5pt}
\begin{center}
    \begin{tikzpicture}
        [scale=.8,auto=left,every node/.style={circle, fill, inner sep=0pt, minimum size=1.5mm, outer sep=0pt},line width=.8mm]
        \node[minimum size=2mm,black!37!white] (u1) at (1,3) [label=left:$u_1$] {};
        \node[minimum size=2mm,black!37!white] (u2) at (4,3) [label=left:$u_2$] {};
        \node[minimum size=2mm,black!37!white] (u3) at (7,3) [label=left:$u_3$] {};
        \node[minimum size=2mm,black!37!white] (u4) at (10,3) [label=left:$u_4$] {};
        \node[minimum size=2mm,black!37!white] (u5) at (13,3) [label=left:$u_5$] {};
        \node[minimum size=2mm,black!37!white] (u6) at (16,3) [label=right:$u_6$]{};
        \node (x1) at (1.3,4) {};  \node (x2) at (1.75,4){};
        \node (x3) at (2.15,4) {};
        \node (x4) at (2.75, 4) {};
        \node (x6) at (3.7,4) {};
        \node (x7) at (4.3,4) {};
        \node (x8) at (6.7,4){};
        \node (x9) at (7.3,4) {}; 
        \node (x10) at (8.3,4){};
        \node (x11) at (8.9, 4) {};
        \node (x12) at (9.7,4) {};
        \node (x14) at (14.25,4) {};
        \node (x15) at (14.75,4) {};
        \node (x16) at (15.7,4) {};
        \node (x17) at (13.3,4){};
        \node (y1) at (1.3,2){}; \node (y2) at (1.8,2){}; \node (y3) at (2.15,2){};\node (y4) at (2.75,2){}; \node (y5) at (3.7,2) {}; \node (y6) at (4.3,2) {}; \node (y7) at (6.7,2) {}; \node (y8) at (7.3,2){}; \node (y9) at (8.3,2) {}; \node (y10) at (8.9,2){}; \node (y11) at (9.7,2) {}; \node (y12) at (10.3,2) {}; \node (y13) at (14.55,2) {}; \node (y14) at (15.7,2){};
        \node (y15) at (15.05,2){};\node (y16) at (14.05,2){};
        \node (y17) at (13.3,2){};
        \node (y18) at (12.7,2){};

        \begin{scope}[on background layer]
        \draw[line width=.4mm] (u1.center) to (x1.center);
        \draw [line width=.8mm] (x1.center)to (x2.center) to (x3.center);
        \draw[line width=.4mm] (x3.center) to (x4.center);
        \draw[line width=.8mm] (x4.center) to (x6.center);
        \draw[line width=.4mm] (x6.center) to (u2.center) to (x7.center);
        \draw[line width=.8mm] (x7.center) to (x8.center);
        \draw[line width=.4mm] (x8.center) to (u3.center) to (x9.center);
        \draw[line width=.8mm] (x9.center)to (x10.center);
        \draw[line width=.4mm] (x10.center) to (x11.center);
        \draw[line width=.8mm](x11.center) to (x12.center);
        \draw[line width=.4mm] (x12.center) to (u4.center);
        \draw[bend left=60, line width=.4mm] (u4.center) to (u5.center);
        \draw[line width=.4mm] (u5.center) to (x17.center);
        \draw[line width=.8mm] (x17.center) to (x14.center);
        \draw[line width=.4mm] (x14.center) to (x15.center);
        \draw[line width=.8mm] (x15.center) to (x16.center);\draw[line width=.4mm](x16.center) to (u6.center);

        \draw [line width=.4mm](u1.center) to (y1.center) to (y2.center); 
        \draw[line width=.8mm] (y2.center) to (y3.center); \draw[line width=.4mm] (y3.center) to (y4.center); \draw[line width=.8mm] (y4.center) to (y5.center); \draw[line width=.4mm](y5.center) to (u2.center) to (y6.center);
        \draw[line width=.8mm] (y6.center) to (y7.center);\draw[line width=.4mm](y7.center) to (u3.center) to (y8.center);\draw[line width=.8mm](y8.center) to (y9.center); \draw[line width=.4mm] (y9.center) to (y10.center); \draw[line width=.8mm] (y10.center) to (y11.center);\draw[line width=.4mm] (y11.center) to (u4.center) to (y12.center);\draw[line width=.8mm](y12.center) to (y18.center);\draw[line width=.4mm](y18.center) to (u5.center) to (y17.center);\draw[line width=.8mm] (y17.center)to (y16.center);
        \draw[line width=.4mm] (y16.center) to (y13.center) to (y15.center);
        \draw[line width=.8mm] (y15.center) to (y14.center);\draw[line width=.4mm](y14.center) to (u6.center);
        \draw[line width=.4mm, line join=bevel] (y1.center) to (x2.center) to (y2.center); 
        \draw[line width=.4mm, dashed] (y3.center) to (x3.center);
        \draw[line width=.4mm] (x3.center) to (y4.center); 
        \draw[line width=.4mm,dashed] (x4.center) to (y4.center);
        \draw[line width=.4mm] (x4.center) to (y3.center);
        \draw[line width=.4mm] (x10.center) to (y10.center); \draw[line width=.4mm] (x11.center) to (y9.center); \draw[line width=.4mm, line join=bevel](y15.center) to (x15.center) to (y13.center) to (x14.center) to (y16.center); 
        \end{scope}
    \end{tikzpicture}
\end{center}
\caption{A chain of $5$ super-clean pinched ladders}
\label{fig:mclchain}
\end{figure}

The following shows that chains of pinched super-clean ladders occur everywhere in $2$-edge-connected graphs.  Thus, it is very natural that chains of pinched super-clean ladders play a significant role in \cref{thm:2econ,thm:inf2econ} below.  But \cref{lem:cleaningblocks} applies to all $2$-edge-connected graphs, not just large or infinite ones, and it may have other applications in future.

\begin{theorem}\label{lem:cleaningblocks}
Let $G$ be a $2$-edge-connected finite or infinite graph with distinct vertices $u$ and $v$.
Then $G$ contains a chain of super-clean pinched ladders with initial vertex $u$ and final vertex $v$ as an induced subgraph.
\end{theorem}
\begin{proof}
Since $G$ is $2$-edge-connected there are two edge-disjoint $uv$-paths.  Let $P$ and $Q$ be two such paths such that (1) $|E(P)|+|E(Q)|$ is minimum, and subject to that (2) $|V(P)\cup V(Q)|$ is minimum.
If one of $P$ or $Q$ is just the edge $uv$, then (1) implies that $P \cup Q$ is an induced cycle, which is the required subgraph.  So we may assume this is not the case, and then (1) implies that $P$ and $Q$ are induced paths.
Vertices of $V(P) \cap V(Q)$ occur along $P$ in the same order as they occur along $Q$, otherwise we could find two paths with fewer edges, contradicting (1).
Let the elements of $V(P)\cap V(Q)$ be $u_1, u_2,\dots, u_{n+1}$ in order along $P$ (or $Q$), so that $u_1=u$ and $u_{n+1}=v$.
Any edge from $P(u_i,u_{i+1})$ to $Q(u_j,u_{j+1})$ has $i=j$, otherwise we could find two paths with fewer edges, contradicting (1).
Let $P_i = P[u_i,u_{i+1}]$, $Q_i = Q[u_i,u_{i+1}]$, and let $L_i$ be the subgraph induced by $V(P_i) \cup V(Q_i)$.  Then $(L_i, P_i, Q_i)$ is a pinched ladder.
A nontrivial cross in $L_i$ allows us to find two paths with fewer edges, contradicting (1).
A nontrivial embedded fan in $L_i$ allows us to find two paths that either have fewer edges, contradicting (1), or have the same number of edges but one fewer vertex, contradicting (2).
Thus, each $L_i$ is a super-clean pinched ladder, and so the subgraph of $G$ induced by $V(P)\cup V(Q)$ is a chain of super-clean pinched ladders, as desired.
\end{proof}

While \cref{lem:cleaningblocks} gives existence of chains of super-clean pinched ladders, it does not require the chain to have at least a particular order, and so we will still need to determine this for \cref{thm:2econ}.
In \cref{ss:highconn} we state a result similar to \cref{lem:cleaningblocks} for $2$-connected graphs.

Our first unavoidability result describes the unavoidable induced subgraphs of finite $2$-edge-connected graphs. Note that for results on finite graphs, these results are existence results.  As such, bounds are chosen for brevity and clarity of proofs, and we make no attempt to optimize our bounds.  Therefore we do not provide an explicit formula for the overall bound $\fref{thm:2econ}(r)$ in \cref{thm:2econ}, although the reader may compute such a formula from the details of our proofs if desired.

\begin{restatable}{theorem}{thmtwoecon}\label{thm:2econ}
For every integer $r\ge 3$, there is an integer $\fref{thm:2econ}(r)$ such that every $2$-edge-connected graph of order at least $\fref{thm:2econ}(r)$ contains $K_r$, an $r$-flower, a super-clean pinched ladder of order at least $r$, a chain of $r$ super-clean pinched ladders, or a member of the family $\Theta_r$ as an induced subgraph.
\end{restatable}

We also prove the infinite analog of \cref{thm:2econ}.
We replace each graph in the above theorem with its a corresponding infinite graph and some additional ladder-like families of graphs.  
%Since graphs in $\mathscr{F}_{\infty}$ and $\mathscr{F}^{\Delta}_\infty$ are not minimally $2$-edge-connected, we replace them by an $\infty$-flower.
%Since an infinite slim ladder (see \cref{fig:infcl}) is not in general minimally $2$-edge-connected, we instead have two possibilities.  
One is an \emph{infinite super-clean pinched ladder}, which is a triple $(L, P, Q)$ such that (1) $L$ is a locally finite graph and $P$ and $Q$ are rays in $L$, (2) $V(L) = V(P) \cap V(Q)$, (2) $P \cap Q$ is a single vertex $\sigma$, the \emph{initial vertex} of $L$, (4) $P$ and $Q$ are induced paths, and (5) relative to the \emph{rails} $P-\{\sigma\}$ and $Q-\{\sigma\}$ there are infinitely many rungs, and all crosses and embedded fans are trivial.  See \cref{fig:inf1pscl}.
The second possibility is a \emph{one-way infinite chain of finite super-clean pinched ladders}, which is a graph $H$ with blocks $L_1, L_2, L_3, \dots$ such that each $L_i$ is a finite super-clean pinched ladder with initial vertex $u_i$ and final vertex $u_{i+1}$.  Thus, $u_2, u_3, u_4, \dots$ are the cutvertices of $H$, which we call \emph{joining vertices}.  We call $u_1$ the \emph{initial} vertex of $H$.  See \cref{fig:infpsclc}.

\begin{figure}[ht]
\centering
       \begin{subfigure}[b]{.45\textwidth}
        \begin{center}
    \begin{tikzpicture}
        [scale=1,auto=left,every node/.style={circle, fill, inner sep=0pt, minimum size=1.75mm, outer sep=0pt},line width=.85mm]
        \node[minimum size=2mm,black!37!white] (u1) at (1,2) [label=left:$\sigma$] {};
        \node (x1) at (1.3,3) {};  \node (x2) at (1.75,3){};
        \node (x3) at (2.15,3) {};
        \node (x4) at (2.75, 3) {};
        \node (y1) at (1.3,1){}; \node (y2) at (1.8,1){}; \node (y3) at (2.15,1){};\node (y4) at (2.75,1){};  
        \begin{scope}[on background layer]
        \draw[line width=.85mm] (u1.center) to (x1.center) to (x2.center) to (x3.center);
        \draw[line width=.35mm] (x3.center) to (x4.center);
        \draw[line width=.85mm] (x4.center) to (4,3); 
        \draw [line width=.85mm](u1.center) to (y1.center); \draw[line width=.35mm] (y1.center) to (y2.center); \draw[line width=.85mm] (y2.center) to (y3.center); \draw[line width=.35mm] (y3.center) to (y4.center); \draw[line width=.85mm] (y4.center) to (4,1); 
        \draw[line width=.35mm, line join=bevel] (y1.center) to (x2.center) to (y2.center); 
        \draw[line width=.35mm, dashed] (y3.center) to (x3.center);
        \draw[line width=.35mm] (x3.center) to (y4.center); 
        \draw[line width=.35mm,dashed] (x4.center) to (y4.center);
        \draw[line width=.35mm] (x4.center) to (y3.center);
        \draw[dotted, line width=.4mm] (3.3,2) to (3.8,2);
        \end{scope}
    \end{tikzpicture}
   \end{center}
    \caption{A one-way infinite super-clean pinched ladder}
    \label{fig:inf1pscl}
    \end{subfigure}
    \begin{subfigure}[b]{.54\textwidth}
       \begin{center}
    \begin{tikzpicture}
        [scale=1,auto=left,every node/.style={circle, fill, inner sep=0pt, minimum size=1.75mm, outer sep=0pt},line width=.85mm]
        \node[minimum size=2mm,black!37!white] (u1) at (1,3) [label=left:$u_1$] {};
        \node[minimum size=2mm,black!37!white] (u2) at (4,3) [label=left:$u_2$] {};
        \node[minimum size=2mm,black!37!white] (u3) at (7,3) [label=left:$u_3$] {};
        \node (x1) at (1.3,4) {};  \node (x2) at (1.75,4){};
        \node (x3) at (2.15,4) {};
        \node (x4) at (2.75, 4) {};
        \node (x6) at (3.7,4) {};
        \node (x7) at (4.3,4) {};
        \node (x8) at (5,4) {};
        \node (x9) at (5.75,4) {};
        \node (x10) at (6.7,4){};
        \node (x11) at (7.3,4) {}; 
        \node (y1) at (1.3,2){}; \node (y2) at (1.8,2){}; \node (y3) at (2.15,2){};\node (y4) at (2.75,2){}; \node (y5) at (3.7,2) {}; \node (y6) at (4.3,2) {}; 
        \node (y7) at (5,2){};
        \node (y8) at (5.6,2){};
        \node (y9) at (6.2,2){};
        \node (y10) at (6.7,2) {};
        \node (y11) at (7.3,2){}; 
        \begin{scope}[on background layer]
        \draw[line width=.85mm] (u1.center) to (x1.center) to (x2.center) to (x3.center);
        \draw[line width=.35mm] (x3.center) to (x4.center);
        \draw[line width=.85mm] (x4.center) to (x6.center) to (u2.center) to (x7.center) to (x8.center);
        \draw[line width=.85mm] (x9.center) to (x10.center) to (u3.center);
        \draw[ line width=.85mm] (u3.center) to (x11.center) to (8.3,4);
        \draw [line width=.85mm](u1.center) to (y1.center); \draw[line width=.35mm] (y1.center) to (y2.center); \draw[line width=.85mm] (y2.center) to (y3.center); \draw[line width=.4mm] (y3.center) to (y4.center); \draw[line width=.85mm] (y4.center) to (y5.center) to (u2.center) to (y6.center) to (y7.center);
        \draw[line width=.85mm] (y9.center) to  (y10.center) to (u3.center) to (y11.center) to (8.3,2);

        \draw[line width=.35mm] (y1.center) to (x2.center); 
        \draw[line width=.35mm, dashed] (y3.center) to (x3.center);
        \draw[line width=.35mm] (x2.center) to (y2.center);
        \draw[line width=.35mm] (x3.center) to (y4.center); 
        \draw[line width=.35mm,dashed] (x4.center) to (y4.center);
        \draw[line width=.35mm] (x4.center) to (y3.center);
        \draw[line width=.35mm] (x8.center) to (y7.center);
        \draw[line width=.35mm] (x8.center) to (y8.center) to (y9.center);
       \draw[line width=.35mm](y7.center) to (y8.center) to (x9.center) to (x8.center);
       \draw[line width=.35mm] (x9.center) to (y9.center);
       \draw[dotted, line width=.4mm] (7.5,3) to (8,3);

        \end{scope}
    \end{tikzpicture}
\end{center}
\caption{A one-way infinite chain of super-clean pinched ladders}
\label{fig:infpsclc}
\end{subfigure}
   \caption{Graphs from \cref{thm:inf2econ}}
   \label{fig:infpscls}
\end{figure}

We also need ladder-like structures, illustrated in \cref{fig:ladders}.
The ladder $L_\infty$ consists of two disjoint induced rays, $P=p_1 p_2\dots$ and $Q=q_1 q_2\dots$, called \emph{rails} and edges $p_iq_i$ for each $i \in \{1,2,3,\dots\}$.
Let $\mathscr{L}_{\infty}$ be the family of graphs obtained from $L_{\infty}$ by subdividing each $p_iq_i$ at least once and each of the rail edges an arbitrary number, possibly zero, of times; see \cref{fig:L1inf}.
Let $\mathscr{L}^{\Delta}_{\infty}$ be the family of graphs obtained from $L_\infty$ by replacing every vertex of one rail with a triangle, and then arbitrarily subdividing each edge not in such a triangle; see \cref{fig:L2inf}.
Let $\mathscr{L}^{\nabla\Delta}_{\infty}$ be the family of graphs obtained from $L_{\infty}$ by replacing every vertex with a triangle, and the arbitrarily subdividing each edge not in such a triangle; see \cref{fig:L3inf}.  
%An \emph{infinite slim ladder} is a triple $(L,P,Q)$ where $L$ is a locally finite graph consisting of two disjoint induced rays $P = p_1 p_2 \dots$ and $Q = q_1 q_2 \dots$ called \emph{rails} and infinitely many edges $p_i q_j$, including $p_1 q_1$, called \emph{rungs}, such that all crosses (defined as for finite ladders) are trivial.  In particular, $L_\infty$ is an infinite slim ladder, and another example is shown in \cref{fig:infcl}.  
%Infinite slim ladders can be regarded as a one-way infinite version of clean ladders.

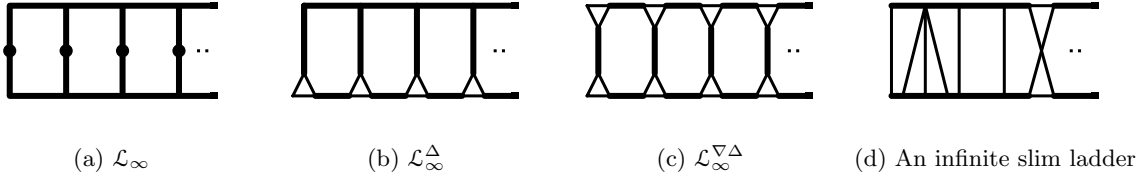
\begin{figure}[!htb]
    \vspace{3pt}
	\begin{subfigure}[b]{.3\textwidth}
		\begin{center}
			\begin{tikzpicture}
				[scale=.6,auto=left,every node/.style={circle, fill, inner sep=0 pt, minimum size=1.75mm, outer sep=0pt},line width=.8mm,line join=round]
				\draw (1,1)--(1,3);
                    \node (1) at (1,2){};
                    \node (2) at (2.25,2){};
                    \node (3) at (3.5,2){};
                    \node (4) at (4.75,2){};
				\draw[line cap=round, -Butt Cap] (1,1)--(5.6,1);
				\draw[line cap=round, -Butt Cap](1,3)--(5.6,3);
				\draw (2.25,1)--(2.25,3);
				\draw (3.5,1)--(3.5,3);
				\draw (4.75,1)--(4.75,3);
				\draw[dotted, line width=.4mm] (5.15,2)--(5.45,2);	
			\end{tikzpicture}
		\end{center}
		\caption{$\mathscr{L}_{\infty}$}
		\label{fig:L1inf}
	\end{subfigure}
	\begin{subfigure}[b]{.3\textwidth}
		\begin{center}
			\begin{tikzpicture}
				[scale=.6,auto=left,every node/.style={circle, fill, inner sep=0 pt, minimum size=.75mm, outer sep=0pt},line width=.8mm,line join=round]
				\draw[dotted, line width=.4mm] (5.45,2)--(5.75,2);
				\draw[line cap=round] (1.25,3)to(1.25,1.5); 
				\draw[line cap=round] (1.5,1)--(2.25,1); \draw[line cap=round] (2.75,1)--(3.5,1); \draw[line cap=round] (4,1)--(4.75,1); 
				\draw[line width=.4mm, line cap=round] (1,1)--(1.5,1)--(1.25,1.5)--(1,1);
				\draw[line width=.4mm] (2.25,1)--(2.75,1)--(2.5,1.5)--(2.25,1);
				\draw[line width=.4mm] (3.5,1)--(4,1)--(3.75,1.5)--(3.5,1);
				\draw[line width=.4mm] (4.75,1)--(5.25,1)--(5,1.5)--(4.75,1);
				\draw[line cap=round,-Butt Cap] (5.25,1) --(6,1);
				\draw[line cap=round, -Butt Cap] (1.25,3)--(6,3);
				\draw [line cap=round](1.25,1.5)--(1.25,3); \draw [line cap=round](2.5,1.5)--(2.5,3);\draw [line cap=round](3.75,1.5)--(3.75,3);\draw[line cap=round] (5,1.5)--(5,3);
			\end{tikzpicture}
		\end{center}
		\caption{$\mathscr{L}^{\Delta}_{\infty}$}
		\label{fig:L2inf}
	\end{subfigure}
	\begin{subfigure}[b]{.3\textwidth}
		\begin{center}
			\begin{tikzpicture}
				[scale=.6,auto=left,every node/.style={circle, fill, inner sep=0 pt, minimum size=.75mm, outer sep=0pt},line width=.8mm, line join=round]
				\draw[dotted, line width=.4mm] (5.45,2)--(5.75,2);
				\draw[line cap=round] (1.5,1)--(2.25,1); \draw[line cap=round] (2.75,1)--(3.5,1); \draw[line cap=round] (4,1)--(4.75,1); 
				\draw[line width=.4mm, line cap=round] (1,1)--(1.5,1)--(1.25,1.5)--(1,1);
				\draw[line width=.4mm] (2.25,1)--(2.75,1)--(2.5,1.5)--(2.25,1);
				\draw[line width=.4mm] (3.5,1)--(4,1)--(3.75,1.5)--(3.5,1);
				\draw[line width=.4mm] (4.75,1)--(5.25,1)--(5,1.5)--(4.75,1);
				\draw[line cap=round,-Butt Cap]  (5.25,1) --(6,1);
				\draw[line cap=round] (1.25,1.5)--(1.25,2.5); \draw [line cap=round](2.5,1.5)--(2.5,2.5);\draw[line cap=round] (3.75,1.5)--(3.75,2.5);\draw[line cap=round] (5,1.5)--(5,2.5);
				\draw [line cap=round](1.5,3)--(2.25,3); \draw[line cap=round] (2.75,3)--(3.5,3); \draw[line cap=round] (4,3)--(4.75,3); 
				\draw[line width=.4mm,line cap=round] (1,3)--(1.5,3)--(1.25,2.5)--(1,3);
				\draw[line width=.4mm] (2.25,3)--(2.75,3)--(2.5,2.5)--(2.25,3);
				\draw[line width=.4mm] (3.5,3)--(4,3)--(3.75,2.5)--(3.5,3);
				\draw[line width=.4mm] (4.75,3)--(5.25,3)--(5,2.5)--(4.75,3);
				\draw[line cap=round,-Butt Cap] (5.25,3)--(6,3);
			\end{tikzpicture}
		\end{center}
		\caption{$\mathscr{L}^{\nabla\Delta}_{\infty}$}
		\label{fig:L3inf}
	\end{subfigure}
	
	\caption{Ladder-like structures}
	\label{fig:ladders}
\end{figure}

Our second main unavoidability result is the following. 

\begin{restatable}{theorem}{thminftwoecon}\label{thm:inf2econ}
Every infinite $2$-edge-connected graph contains one of the following as an induced subgraph: 
$K_{\infty}$, an $\infty$-flower, a one-way infinite chain of finite super-clean pinched ladders, a one-way infinite super-clean pinched ladder, or a member of the family $\Theta_\infty \cup \mathscr{L}_{\infty}\cup \mathscr{L}^{\Delta}_{\infty} \cup \mathscr{L}^{\nabla\Delta}_{\infty}$.
\end{restatable}

In \cref{sec:prelimresults}, we state some results we need for our proofs.  In \cref{sec:finite,sec:infinite}, we prove \cref{thm:2econ,thm:inf2econ}, respectively.
\cref{sec:conclusion} contains some final remarks.
In the \hyperref[sec:otherstructures]{Appendix}, we use \cref{thm:2econ,thm:inf2econ} to describe other unavoidable substructures for $2$-edge-connected graphs.

\section{Preliminary results}\label{sec:prelimresults}

\subsection{Basic Results}\label{sec:basicresults}

In this section we state some results we will use later.
In \cref{sec:prelimresults,sec:finite,sec:infinite}, if we say that \emph{$G$ contains $H$} without further qualification, we mean that $G$ contains $H$ as an induced subgraph.
We begin with two simple results.

\begin{lemma}\label{lem:fanflower}
If the apex vertex of a fan $F$ has degree at least $3r-1$ then $F$ has an $r$-flower as an induced subgraph.
\end{lemma}

\begin{proof}
Suppose the apex of $F$ is $u$, the rim is $Q$, and $w_1, w_2\dots , w_{3r-1}\in V(Q)$ are the endpoints of $3r-1$ edges incident with $u$, in order along $Q$.  The subgraph of $L$ induced by $u$ and the sets $V(Q[w_{3i+1},w_{3i+2}])$ for $0\le i\le r-1$ is an $r$-flower.
\end{proof}

\begin{observation}\label{obs:degdiam}
If a connected graph $G$ has maximum degree at most $d \ge 2$ and diameter at most $k \ge 1$, then the order of $G$ is at most $f_{0}(d,k) = 1 + d + d(d-1) + d(d-1)^2 + \dots + d(d-1)^{k-1}$.  Therefore, for $d \ge 3$ and $k \ge 2$, every connected graph $G$ with order at least $\fref{obs:degdiam}(d,k) = f_{0}(d-1,k-1)+1$ either has maximum degree at least $d$ or has diameter at least $k$.
\end{observation}

\subsection{Unavoidable $2$-connected induced subgraphs}\label{ss:highconn}

In this subsection we describe the known unavoidable induced subgraphs for $2$-connected graphs, which are instrumental in proving our results.  We need a number of additional definitions before stating the results.

In \cite{unavoidableinducedsubgraphs}, the first author, Ding, and Oporowski considered the induced subgraph ordering for $2$-connected graphs. 
They defined two families of graphs that generalize stars and paths to $2$-connected graphs.
The first family $\Theta_r$, where $r \ge 3$ is an integer or $r = \infty$, were defined in \cref{ss:mainresults}.  

The description of the second family requires some preliminary definitions.
A \emph{ladder} is a triple $(L,P,Q)$ that consists of a graph $L$ whose vertices all lie on two disjoint induced paths $P=p_1p_2 \dots p_{\ell}$ and $Q=q_1q_2 \dots q_m$ of $L$, called \emph{rails}, and where $p_1q_1,p_{\ell}q_m\in E(L)$.
Rungs, cross, span, and trivial are defined analoguously to pinched ladders in \cref{ss:mainresults}.

We define a \emph{clean ladder} (\emph{clean pinched ladder}) to be a ladder (pinched ladder) $L$ where all crosses are trivial and $|V(L)| \ge 3$. Clean ladders are $2$-connected. Let $\Lambda_r$ be the set of all clean ladders of order at least $r$, which can be considered as an analog of the path $P_r$ for $2$-connected graphs.
Every cycle of length at least $r$ belongs to $\Lambda_r$.  Note that graphs in $\Lambda_r$ are in general not minimally $2$-edge-connected (in the induced subgraph ordering).

We observe that there is a result similar to \cref{lem:cleaningblocks} for 2-connected graphs, which shows that induced clean ladders are ubiquitous in $2$-connected graphs.  This result does not seem to have been noted before, although clean ladders play a key role in the results on unavoidable large $2$-connected induced subgraphs, described later in this subsection.  We hope this result will be useful for future applications.

\begin{theorem}\label{thm:2conpscl}
Let $G$ be a $2$-connected finite or infinite graph with distinct vertices $u$ and $v$.
Then $G$ contains an induced clean pinched ladder with initial vertex $u$ and final vertex $v$.
This may be regarded as an induced clean ladder in which each of $u$ and $v$ is an end of a rail.
\end{theorem}
\begin{proof}
Since $G$ is $2$-connected there are two internally disjoint $uv$-paths.  Let $P$ and $Q$ be two such paths such that $|V(P)\cup V(Q)|$ is minimum.
If one of $P$ or $Q$ is just the edge $uv$, then minimality implies that $P \cup Q$ is an induced cycle, which is the required subgraph.  So we may assume this is not the case, and then the minimality of $|V(P)\cup V(Q)|$ implies that $P$ and $Q$ are induced paths.
Let $L$ be the subgraph of $G$ induced by $V(P)\cup V(Q)$.
A nontrivial cross in $L$ allows us to find two paths with fewer vertices, contradicting minimality.
Thus,  $L$ is a clean pinched ladder, as desired.  Also, $L$ may be regarded as a clean ladder by choosing two distinct edges of $P \cup Q$, one incident with $u$ and one with $v$, to be the end rungs.
\end{proof}

The 2-connected analog of \cref{thm:2econ} as shown in \cite{unavoidableinducedsubgraphs} is as follows.

\begin{theorem}[{\cite[(1.4)]{unavoidableinducedsubgraphs}}]\label{thm:2con}
For every integer $r\ge 3$, there is an integer $\fref{thm:2con}(r)$ such that every $2$-connected graph of order at least $\fref{thm:2con}(r)$ contains  $K_r$ or a member of the family $\Theta_r\cup\Lambda_r$  as an induced subgraph.
\end{theorem}

In \cite{unavoidableinfinducedsubgraphs}, the first author, Ding, and Oporowski extended this theorem to infinite $2$-connected graphs.  To do this, one of the families we need is $\Theta_\infty$. 
\cref{fig:theta} shows the structure of graphs in this family; the graphs in \cref{fig:K2inf}
form the class $\mathscr{K}_{2,\infty}$ consisting of all subdivisions of $K_{2,\infty}$.

\begin{figure}[!htb]
	\begin{subfigure}[b]{.47\textwidth}
		\begin{center}
			\begin{tikzpicture}
			[scale=.8,auto=left,every node/.style={circle, fill, inner sep=0 pt, minimum size=.75mm, outer sep=0pt},line width=.8mm, line join=round,line cap=bevel]
			\draw[dotted, line width=.4mm] (3.75,1)--(4.15,1);
			\draw (2.15,3)--(4.75,1)--(3.45,3)--(2.25,1) -- (2.15,3)--(1,1)--(3.45,3)--(3.5,1)--(2.15,3)--(2.25,1); 
			  \node[minimum size=2mm,black] (u1) at (2.15,3) [label=left:$u_1$] {};
            \node[minimum size=2mm,black] (u2) at (3.45,3) [label=right:$u_2$] {};
		\end{tikzpicture}
	\end{center}
	\caption{$\mathscr{K}_{2,\infty}$, graphs with no edge $u_1u_2$}
	\label{fig:K2inf}
	\end{subfigure}
\begin{subfigure}[b]{.47\textwidth}
\begin{center}
\begin{tikzpicture}
	[scale=.8,auto=left,every node/.style={circle, fill, inner sep=0 pt, minimum size=.75mm, outer sep=0pt},line width=.8mm,line join=round,line cap=bevel]
	\draw[dotted, line width=.4mm] (3.75,1)--(4.15,1);
			\draw (2.15,3)--(4.75,1)--(3.45,3)--(2.25,1) -- (2.15,3)--(1,1)--(3.45,3)--(3.5,1)--(2.15,3)--(2.25,1); 
	\draw[line width=.4mm] (2.15,3)--(3.45,3);
			  \node[minimum size=2mm,black] (u1) at (2.15,3) [label=left:$u_1$] {};
            \node[minimum size=2mm,black] (u2) at (3.45,3) [label=right:$u_2$] {};
\end{tikzpicture}
\end{center}
\caption{Graphs with edge $u_1u_2$}
\label{fig:K2infplus}
\end{subfigure}
\caption{Structure of $\Theta_{\infty}$}
\label{fig:theta}
\end{figure}
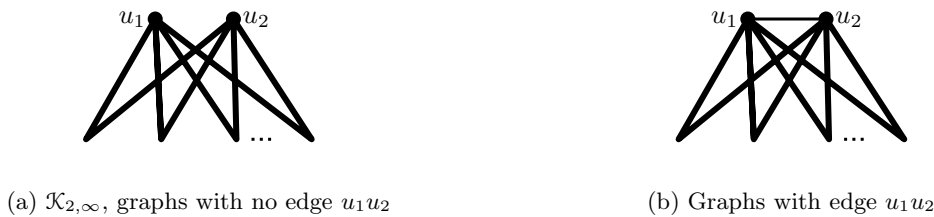

For several graph classes below we use the operation of \emph{replacing a vertex by a triangle}: if $v$ has degree $d \le 3$ and neighbors $u_i$, $1 \le i \le d$, we delete $v$, add a triangle of new vertices $(v_1 v_2 v_3)$, and add edges $u_i v_i$ for $1 \le i \le d$.

We use certain infinite fan-like structures and ladders.  Let $F_\infty$ denote the graph that consists of a vertex $u$, a ray $v_1 v_2 v_3 \ldots$ (called the \emph{rim}), and edges $uv_i$ for all $i \in \{1, 2, 3, \dots\}$.
Let $\mathscr{F}_{\infty}$ be the family of all subdivisions of $F_\infty$; see \cref{fig:famFinf}.
Let $\mathscr{F}_\infty^{\Delta}$ be the family of graphs obtained from $F_{\infty}$ by replacing every rim vertex with a triangle, and then arbitrarily subdividing each edge not in such a triangle; see \cref{fig:Fdeltinf}.
An \emph{infinite slim ladder} is a triple $(L,P,Q)$ where $L$ is a locally finite graph consisting of two disjoint induced rays $P = p_1 p_2 \dots$ and $Q = q_1 q_2 \dots$ called \emph{rails} and infinitely many edges $p_i q_j$, including $p_1 q_1$, called \emph{rungs}, such that all crosses (defined as for finite ladders) are trivial.  In particular, $L_\infty$ is an infinite slim ladder, and another example is shown in \cref{fig:infcl}.  
Infinite slim ladders can be regarded as a one-way infinite version of clean ladders.

\begin{figure}[!htb]
\centering
\begin{subfigure}[b]{.3\textwidth}
		\begin{center}
			\begin{tikzpicture}
				[scale=.9,auto=left,every node/.style={circle, fill, inner sep=0 pt, minimum size=2mm, outer sep=0pt},line width=.8mm, line join=round, line cap=bevel]
				\draw[dotted, line width=.4mm] (4.95,2)--(5.35,2);
				\draw (2.25,1)--(2.65,3)--(1,1); 
                \draw(2.65,3)--(3.5,1);\draw (2.65,3)--(4.75,1);
				\draw(1,1)--(5,1);
                \node (u) at (2.65,3) [label={[label distance=2pt]above:$u$}]{};
                \node (v1) at (1,1) [label=below:$v_1$]{};
                \node (v2) at (2.25,1)[label=below:$v_2$]{};
                 \node (v3) at (3.5,1)[label=below:$v_3$]{};
                  \node (v3) at (4.75,1)[label=below:$v_4$]{};
			\end{tikzpicture}
		\end{center}
		\caption{$\mathscr{F}_{\infty}$}
		\label{fig:famFinf}
	\end{subfigure}
    \kern20pt
	\begin{subfigure}[b]{.3\textwidth}
		\begin{center}
			\begin{tikzpicture}
				[scale=.9,auto=left,every node/.style={circle, fill, inner sep=0 pt, minimum size=2mm, outer sep=0pt},line width=.8mm, line join=round, line cap=bevel]
				\draw[dotted, line width=.4mm] (5.45,2)--(5.85,2);
				\draw[line cap=round] (2.5,1.5)--(2.65,3)to (1.25,1.5); \draw[line cap=round] (3.75,1.5)--(2.65,3)--(2.5,1.5);\draw[line cap=round] (2.65,3)--(5,1.5);
				\draw[line cap=round] (1.5,1)--(2.25,1); \draw [line cap=round](2.75,1)--(3.5,1); \draw[line cap=round] (4,1)--(4.75,1); 
				\draw[line width=.4mm, line cap=round] (1,1)--(1.5,1)-- (1.25,1.5)--(1,1);
				\draw[line width=.4mm] (2.25,1)--(2.75,1)--(2.5,1.5)--(2.25,1);
				\draw[line width=.4mm] (3.5,1)--(4,1)--(3.75,1.5)--(3.5,1);
				\draw[line width=.4mm] (4.75,1)--(5.25,1)--(5,1.5)--(4.75,1);
				\draw[line cap=round, -Butt Cap] (5.25,1)--(5.5,1);
                \node (u) at (2.65,3) [label={[label distance=2pt]above:$u$}]{};
                \node (v1) at (1.25,1.5)[label=left:$v_1$]{};
                \node (w1) at (1,1) [label=below:$w_1$]{};
                \node (x1) at (1.5,1)[label=below:$x_1$]{};
                \node (v2) at (2.5,1.5)[label=left:$v_2$]{};
                \node (w2) at (2.25,1) [label=below:$w_2$]{};
                \node (x2) at (2.75,1)[label=below:$x_2$]{};
                \node (v3) at (3.75,1.5)[label=left:$v_3$]{};
                \node (w3) at (3.5,1) [label=below:$w_3$]{};
                \node (x3) at (4,1)[label=below:$x_3$]{};
                \node (v4) at (5,1.5)[label=left:$v_4$]{};
                \node (w4) at (4.75,1) [label=below:$w_4$]{};
                \node (x4) at (5.25,1)[label=below:$x_4$]{};
			\end{tikzpicture}
            
		\end{center}
		\caption{$\mathscr{F}_{\infty}^{\Delta}$}
		\label{fig:Fdeltinf}
	\end{subfigure}
    \kern20pt
    \begin{subfigure}[b]{.3\textwidth}
		\begin{center}
			\begin{tikzpicture}
				[scale=.9,auto=left,every node/.style={circle, fill, inner sep=0 pt, minimum size=2mm, outer sep=0pt}, line width=.4mm]
				\draw[line width=.8mm,line cap=round](1,1) to (4.05,1); \draw[line cap=round] (4.05,1)--(4.6,1);\draw[line width=.8mm, line cap=round, -Butt Cap] (4.6,1)--(5.6,1); \draw[line width=.8mm, line cap=round](1,3) to (4.05,3);\draw (4.05,3)--(4.6,3);\draw[line width=.8mm, line cap=round, -Butt Cap]  (4.6,3) to (5.6,3);\draw[line cap=round](1,1)--(1,3);
				\draw[line join=bevel](1.75,1)-- (1.75,3) --(1.25,1);\draw(1.75,3)--(2.25,1);
				\draw (2.5,3) to (2.5,1);\draw (3.5,1)--(3.5,3);
				\draw (4.05,3) to (4.6,1);\draw(4.6,3) to (4.05,1);
				\draw[dotted] (4.95,2)--(5.25,2);
                %\node[fill=white] (a) at (1,0.5){};
                \node (p1) at (1,3)[label=above:$p_1$]{};
                \node (q1) at (1,1)[label=below:$q_1$]{};
			\end{tikzpicture}
		\end{center}
		\caption{An infinite slim ladder}
		\label{fig:infcl}
	\end{subfigure}
\caption{Two types of fan-like structure and an infinite slim ladder}
\label{fig:fans}
\end{figure}

Note that the labeling of vertices in \cref{fig:famFinf,fig:Fdeltinf} will be used in \cref{sec:infinite}.

\begin{theorem}[{\cite[(1.7)]{unavoidableinfinducedsubgraphs}}]\label{thm:inf2con}
Every infinite $2$-connected graph contains $K_{\infty}$, an infinite slim ladder, or a member of $\Theta_\infty\cup\mathscr{F}_\infty\cup\mathscr{F}^{\Delta}_\infty\cup\mathscr{L}_\infty\cup\mathscr{L}^{\Delta}_{\infty}\cup\mathscr{L}^{\nabla\Delta}_{\infty}$ as an induced subgraph.
\end{theorem}

As shown in \cite{unavoidableinfinducedsubgraphs}, \cref{thm:2con} on finite graphs can also be proved using \cref{thm:inf2con} on infinite graphs.
We will use \cref{thm:2con,thm:inf2con} to prove our results on $2$-edge-connected graphs, \cref{thm:2econ,thm:inf2econ}.

\section{Proof of the finite theorem}\label{sec:finite}
In this section, we prove \cref{thm:2econ} using \cref{lem:cleaningblocks,thm:2con}.

\begin{lemma}\label{lem:largecutverttree}
For all integers $p,q\ge 3$ there exists an integer $\fref{lem:largecutverttree}(p,q)$ such that every $2$-edge-connected graph of order at least $\fref{lem:largecutverttree}(p,q)$ has a block-cutvertex tree of order at least $p$ or a block of order at least $q$.
\end{lemma}

\begin{proof}
Suppose the graph $G$ has blocks $B_1, B_2, \dots, B_k$.  Then $|V(G)| = 1 + \sum_{i=1}^k (|V(B_i)|-1)$, so if $k \le p-1$ and $|V(B_i)| \le q-1$ for each $i$ we have at most $1+(p-1)(q-2)$ vertices.  Thus, $|V(G)| \ge \fref{lem:largecutverttree}(p,q)=2+(p-1)(q-2)$ guarantees that either $k \ge p$ or $|V(B_i)| \ge q$ for some $i$.
\end{proof}

\cref{lem:largecutverttree} implies that we can consider two cases: either $G$ has a large block or $G$ has a large block-cutvertex tree.  In the former case we apply \cref{thm:2con,lem:super-cleaning}, and we address the latter in \cref{lem:pathtobc,lem:cleaningblocks,lem:bctopsclc}.

We can now prove that if $(L,P,Q)$ is a large clean ladder, then $L$ contains either an $r$-flower, a large super-clean pinched ladder, or a large chain of super-clean pinched ladders as an induced subgraph. 
\begin{lemma}\label{lem:super-cleaning}
Let $r\ge 3$ be an integer.  If $(L,P,Q)$ is a clean ladder of order at least $\fref{lem:super-cleaning}(r)$, then $L$ contains a super-clean pinched ladder of order at least $r$, a chain of $r$ super-clean pinched ladders, or an $r$-flower as an induced subgraph.
\end{lemma}
\begin{proof}
If $r=3$, then $\fref{lem:super-cleaning}(3) = 3$ satisfies the conditions, so we may assume that $r \ge 4$.
A chain of at most $r-1$ pinched super-clean ladders each of order at most $r-1$ has order at most $(r-1)(r-2)+1$.  Thus, such a chain with order at least $(r-1)(r-2)+2$ either has at least $r$ pinched ladders or contains a pinched ladder of order at least $r$.
Also, if a chain of pinched super-clean ladders has a vertex of degree $d+2$ for some $d \ge 3$, then it contains an embedded fan with an apex vertex of degree $d$.

So suppose that $L$ has order at least $\fref{lem:super-cleaning}(r)=\fref{obs:degdiam}(3r+1,(r-1)(r-2)+1)$.  Then there are two possibilities.  First, $L$ has a vertex of degree at least $3r+1$, in which case it contains a fan with an apex vertex of degree at least $3r-1$, and hence an $r$-flower by \cref{lem:fanflower}.  Otherwise, $L$ has diameter at least $(r-1)(r-2)+1$; take two vertices $u$ and $v$ at this distance in $L$.  Then by \cref{lem:cleaningblocks} $L$ contains a chain of pinched super-clean ladders $L'$ with initial vertex $u$ and final vertex $v$, with order at least $(r-1)(r-2)+2$ since $L'$ has a path from $u$ to $v$.  Therefore, $L'$, and hence $L$, contains either at least $r$ pinched ladders or a pinched ladder of order at least $r$.
 \end{proof}

A long path in a block-cutvertex tree means that the graph has a long chain of blocks.

\begin{lemma}\label{lem:pathtobc}
Let $t$ be a positive integer. If the block-cutvertex tree $T$ of a graph $G$ has a path of order $\fref{lem:pathtobc}(t)$, then $G$ contains an induced chain of at least $t$ blocks.
\end{lemma}

\begin{proof}
Let $\fref{lem:pathtobc}(t)=2t$.
Let $P$ be a path of order $\fref{lem:pathtobc}(t)$ in $T$.  Then alternate vertices of $P$ represent blocks of $G$. So there are at least $t$ vertices of $P$ representing blocks of $G$.  A minimal subpath of $P$ containing $t$ vertices representing blocks corresponds to a induced chain of $t$ blocks in $G$.
\end{proof}

\begin{lemma}\label{lem:bctopsclc}
Let $t$ be a positive integer.  If a $2$-edge-connected graph $G$ contains an induced subgraph that is a chain of $t$ blocks, then $G$ contains a chain of at least $t$ super-clean pinched ladders as an induced subgraph.
\end{lemma}
\begin{proof}
Let $G'$ be an induced chain of $t$ blocks in $G$.  Choose a vertex $u$ in the first block of $G'$ that is not a cutvertex of $G'$, and a similar vertex $v$ in the last block of $G'$.  By \cref{lem:cleaningblocks}, $G'$, and hence $G$, has an induced subgraph $H$ that is a chain of super-clean pinched ladders with initial vertex $u$ and final vertex $v$.  Each block of $H$ lies inside a block of $G'$, 
and $H$ has an edge of every block of $G'$ since $H$ has a $uv$-path.  Hence $H$ has at least $t$ blocks.
\end{proof}

We can now prove the finite theorem.
\thmtwoecon*

\begin{proof}
Let $n_1=\fref{lem:super-cleaning}(r)$, $n_2=\max\{\fref{lem:pathtobc}(r),\fref{thm:2con}(n_1)\}$, $n_3=\fref{thm:indconnfin}(n_2)$, and $\fref{thm:2econ}(r)=\fref{lem:largecutverttree}(n_3,n_2)$.
Then $\fref{thm:2econ}(r)\ge n_3\ge n_2\ge n_1\ge r$.
Note that $n_2$ contributes to $\fref{thm:2econ}(r)$ in two ways because we can get a large block in two ways: as a large block initially, or as a large block coming from a block vertex of high degree in the block-cutvertex tree.
Let $G$ be a $2$-edge-connected graph of order at least $\fref{thm:2econ}(r)$, and let $T$ be the block-cutvertex tree of $G$.

\cref{lem:largecutverttree} implies that either $G$ has a block of order at least $n_2$ or $T$ has order at least $n_3$.
If a block has order at least $n_2$, then \cref{thm:2con} implies that $G$ contains an induced member of $\{K_{n_1}\} \cup \Theta_{n_1} \cup \Lambda_{n_1}$.  Then $G$ contains as an induced subgraph either a member of $\{K_r\} \cup \Theta_r$ and the conclusion holds, or a member of $\Lambda_{n_1}$, in which case \cref{lem:super-cleaning} implies that $G$ contains a super-clean pinched ladder of order at least $r$, a chain of $r$ super-clean pinched ladders, or an $r$-flower, as desired. 

We may therefore assume that $G$ does not have a block of order at least $n_2$, so $T$ has order at least $n_3$.
\cref{thm:indconnfin} implies that $T$ contains as an induced subgraph either $K_{1,n_2}$ or a path of order at least $n_2$.
Suppose that $T$ contains an induced $K_{1,n_2}$.  
Since $G$ does not have a block of order at least $n_2$, it follows that the vertex, say $v$, of high degree in $T$ is a cutvertex of $G$.
So $v$ is in at least $n_2$ blocks.
Since $G$ is $2$-edge-connected, it follows that each block contains an induced cycle that includes $v$.
Thus $G$ contains an $n_2$-flower, and thus an $r$-flower, and the conclusion follows.

We may therefore assume that $T$ has a path of order at least $n_2\ge \fref{lem:pathtobc}(r)$. 
\cref{lem:pathtobc} implies that $G$ contains an induced chain of $r$ blocks.
Then \cref{lem:bctopsclc} implies that $G$ contains an induced chain of at least $r$ super-clean pinched ladders, and the conclusion follows.
\end{proof}

\section{Proof of the infinite theorem}\label{sec:infinite}
In this section we prove \cref{thm:inf2econ} using \cref{lem:cleaningblocks,thm:inf2con}.

\begin{lemma}\label{lem:infblock}
Let $G$ be an infinite $2$-edge-connected graph.  Then $G$ has an induced infinite block, a vertex that is in infinitely many blocks,  or an induced one-way infinite chain of finite blocks.
\end{lemma}
\begin{proof}
If $G$ contains an infinite block, then the conclusion follows.
We may therefore assume that $G$ does not contain an infinite block.  
Since $G$ is infinite, it follows that $G$ contains infinitely many finite blocks. 
Thus, the block-cutvertex tree $T$ of $G$ is infinite, connected, and every vertex of infinite degree corresponds to a cutvertex of $G$.
Since $T$ is bipartite, connected, and infinite, \cref{thm:indconninf} implies that $T$ contains as an induced subgraph either $K_{1,\infty}$ or an induced ray.

If $T$ contains an induced $K_{1,\infty}$, then $G$ has a cutvertex that is in infinitely many blocks, and the conclusion follows.
We may therefore assume that $T$ contains an induced ray.
This ray corresponds to an induced one-way infinite chain of finite blocks in $G$, and the conclusion follows.\end{proof}

Since members of $\mathscr{F}_{\infty}$ and $\mathscr{F}^{\Delta}_{\infty}$ are not minimally $2$-edge-connected, we show that each member contains an induced $\infty$-flower.

\begin{lemma}\label{lem:inffan}
Each member of one of the families $\mathscr{F}_{\infty}$ or $\mathscr{F}_\infty^\Delta$ contains an induced $\infty$-flower.
\end{lemma}
\begin{proof}
Suppose $G$ is a member of $\mathscr{F}_{\infty}$, labeled as in \cref{fig:famFinf}.
Let $R$ be the rim of $G$ and for each positive integer let $Q_i$ be the $uv_i$-path all of whose interior vertices have degree $2$.
The subgraph of $G$ induced by $\bigcup_{i=0}^{\infty} V(Q_{3i+1} \cup R[v_{3i+1},v_{3i+2}]\cup Q_{3i+2})$ is an $\infty$-flower.

We may therefore assume that $G$ is a member of $\mathscr{F}_\infty^\Delta$, labeled as in \cref{fig:Fdeltinf}.
For each positive integer let $Q_i$ be the $uv_i$-path all of whose interior vertices have degree $2$, and let $R$ be the path $G - \bigcup_{i=1}^\infty V(Q_i)$.
The subgraph of $G$ induced by $\bigcup_{i=0}^\infty V(Q_{2i+1} \cup R[x_{2i+1}, w_{2i+2}] \cup Q_{2i+2})$ is an $\infty$-flower.
\end{proof}

\begin{lemma}\label{lem:infcleanladder}
An infinite slim ladder contains either a one-way infinite chain of finite super-clean pinched ladders or a one-way infinite super-clean pinched ladder as an induced subgraph.
\end{lemma}

\begin{proof}
Suppose the infinite slim ladder is $(L,P,Q)$ with $P=p_1 p_2 \dots$ and $Q = q_1 q_2 \dots$.  First assume that $L$ has only finitely many nontrivial embedded fans.  Then we may choose a rung $p_i q_j$ so that all vertices of nontrivial embedded fans occur either on $P$ before $p_i$, or on $Q$ before $q_j$.  Deleting $\{p_1, p_2, \dots p_{i-1}\} \cup \{q_1, q_2, \dots, q_{j-1}\}$ therefore gives an infinite slim ladder $L'$ with no nontrivial embedded fans.  If some $v \in \{p_i, q_j\}$ has degree $2$ in $L'$ then $L'$ is a one-way infinite super-clean pinched ladder with initial vertex $v$.  Otherwise there is a cross $p_i q_{j+1}, p_{i+1} q_j$ and $L'-q_j$ is a one-way infinite super-clean pinched ladder with initial vertex $p_i$.

Now suppose that $L$ has infinitely many nontrivial embedded fans.  We may assume that infinitely many of these have an apex vertex on $P$.  Then we may choose an infinite sequence of vertex-disjoint nontrivial embedded fans of $L$ with apex on $P$ and not including $p_1$ or $q_1$.  Let $p_1', p_2', \dots$ be the sequence of apex vertices of these fans in order along $P$, where $p_i'$ is the apex vertex of a fan with rim $Q[q_i', q_i'']$.  Let $p_0' = p_1$ and $q_0'' = q_1$.  Then deleting the internal vertices of all paths $Q[q_i', q_i'']$ gives a one-way infinite chain of block $B_1, B_2, \dots$.
Each block $B_i$ has a spanning cycle consisting of two paths: $P[p'_i, p'_{i+1}]$, and a path formed by $p_i' q_i''$, $Q[q_i'', q_{i+1}']$, and $p_{i+1}' q_{i+1}'$.  Blocks $B_i$ and $B_{i+1}$ are separated by cutvertex $p'_i$.  By \cref{lem:cleaningblocks} each $B_i$ contains a chain $L_i$ of pinched super-clean ladders with initial vertex $p'_{i-1}$ and final vertex $p'_i$, and $\bigcup_{i=1}^\infty L_i$ is a one-way infinite chain of pinched super-clean ladders with initial vertex $p_0'=p_1$.
\end{proof}

We can now prove \cref{thm:inf2econ}.

\thminftwoecon*
\begin{proof}
\cref{lem:infblock} implies that $G$ has either an infinite block, a vertex that is in infinitely many blocks, or a one-way infinite chain of finite blocks.

If $G$ contains an infinite block, then \cref{thm:inf2con} implies that $G$ contains one of the following as an induced subgraph:  $K_{\infty}$, an infinite slim ladder, or a member of one of the following families: $\Theta_\infty$, $\mathscr{F}_{\infty}$, $\mathscr{F}_{\infty}^{\Delta}$, $\mathscr{L}_{\infty}$, $\mathscr{L}^{\Delta}_{\infty}$, or $\mathscr{L}^{\nabla\Delta}_{\infty}$.
If $G$ contains $K_\infty$ or a member of $\Theta_\infty$, $\mathscr{L}_{\infty}$, $\mathscr{L}^{\Delta}_{\infty}$, or $\mathscr{L}^{\nabla\Delta}_{\infty}$, then the theorem holds.
If $G$ contains a member of $\mathscr{F}_{\infty}$ or $\mathscr{F}_{\infty}^\Delta$, then \cref{lem:inffan} implies that $G$ contains an $\infty$-flower, and the conclusion follows.
If $G$ contains an infinite slim ladder, then \cref{lem:infcleanladder} implies that $G$ contains either an infinite chain of finite super-clean pinched ladders or a one-way infinite super-clean pinched ladder, and the conclusion follows.

We may therefore assume that $G$ does not contain an infinite block.
Suppose that $G$ has a cutvertex $v$ that is in infinitely many blocks $B_1$, $B_2$, \dots.   
Let $C_i$ be a shortest cycle in $B_i$ that includes $v$.  Then $\bigcup\limits_{i\in\mathbb{N}} C_i$ is an $\infty$-flower, and the conclusion follows.

We may therefore assume that $G$ is locally finite and has a one-way infinite chain of finite blocks, $B'_1$, $B'_2$, \dots. 
Since each block is finite, we may use \cref{lem:cleaningblocks} in a way similar to the second part of the proof of \cref{lem:infcleanladder}, to find a chain of super-clean ladders between the joining vertices of each block $B'_i$ for $i \ge 2$.  The union of these is a one-way infinite chain of finite super-clean pinched ladders, and the conclusion follows.
\end{proof}

\section{Conclusion}\label{sec:conclusion}
The next natural question would be to determine the unavoidable induced subgraphs, or even unavoidable subgraphs, for $3$-edge-connected or $3$-connected graphs.  However, these sets cannot have a simple structure.  
Consider an arbitrary $3$-connected (equivalently $3$-edge-connected) cubic graph $G$.
The deletion of a vertex or edge in $G$ results in a graph that is not $3$-edge-connected because there would be at least one vertex of degree two in $G- v$.
Thus, the list of unavoidable $3$-edge-connected (or $3$-connected) induced subgraphs or subgraphs must include every large cubic graph.

In the \hyperref[sec:otherstructures]{Appendix} below we use our results for induced subgraphs to derive unavoidable substructure results for $2$-edge-connected graphs under several other common orderings.  There is a further ordering of graphs that is a weakening of the induced subgraph, subgraph and topological minor orderings, but incomparable with the minor ordering, namely the \emph{(weak) immersion} ordering.  For immersions it makes most sense to deal with multigraphs and consider edge-connectivity rather than (vertex\nobreakdashes-)connectivity.  It is easy to show that the only large unavoidable immersions for $2$-edge-connected multigraphs are cycles $C_r$ (this follows from our \crefWithTheorem{mgtopminorfin}).  Barnes \cite{mattthesis} determined the unavoidable immersions for $3$-edge-connected multigraphs, and Ding and Qualls \cite{unavoidableimmersions4econ} determined them for $4$-edge-connected multigraphs.  Ding and Qualls (see \cite[Section 3.5]{brittthesis}) have also determined the unavoidable immersions for $2$- and $3$-edge-connected infinite multigraphs.  We note that the unavoidable immersions for $2$-edge-connected infinite multigraphs (namely $S_{2,\infty}$, $P_{2,\infty}$, and $L_\infty$) follow from \crefWithTheorem{mgtopminorinf}.

Unavoidable substructure results for matroids, or involving orderings other than those we have already discussed, or involving different notions of connectivity, are also known, and we mention a few of these.  Ding, Oxley, Oporowski, and Vertigan \cite{Unavoidableminors3connbinmatroids,Unavoidableminors3connmatroids} determined the unavoidable large minors for $3$-connected binary and general matroids, respectively.
There is a \emph{parallel minor} ordering for both graphs and matroids that strengthens the minor ordering in a different direction from topological minors.
C.~Chun, Ding, Oporowski, and Vertigan \cite{unavoidableparminor4conngraphs} determined the unavoidable parallel minors for $k$-connected graphs for $k \le 3$, and for internally $4$-connected graphs, and C.~Chun and Oxley  \cite{Unavoidableparminorregmatroids} determined the unavoidable parallel minors for $3$-connected regular matroids.
C.~Chun and Ding \cite{CD10} obtained results on large unavoidable topological minors and parallel minors in infinite graphs based on the idea of `loose connectivity'.

\appendix
\section*{Appendix: Other unavoidable substructures}\label{sec:otherstructures}
\gdef\thesection{A}

In this appendix we apply our results on unavoidable induced subgraphs (\cref{thm:2econ,thm:inf2econ}) to provide results on other unavoidable substructures in $2$-edge-connected graphs.  The proofs are mostly straightforward, so we leave them to the reader, giving only occasional comments.

\subsection{Unavoidable subgraphs, topological minors, and minors}\label{ss:unavoidsub}

In this subsection we state results on unavoidable subgraphs, topological minors, and minors in $2$-edge-connected graphs.  We assume the reader is familiar with these orderings.  Our results also imply the existence of unavoidable Eulerian subgraphs.

Let $\mathscr{K}_{2,r}$ be the subset of $\Theta_r$ consisting of subdivisions of $K_{2,r}$, and we define $\mathscr{K}_{2,\infty}$ similarly.  An $r$-flower or $\infty$-flower is \emph{triangular} if each of its cycles is a triangle. \cref{thm:2econ} immediately implies the following.

\begin{theorem}\label{unavsgtm}
For every integer $r \ge 3$, there is an integer $\fref{unavsgtm}(r)$ such that every $2$-edge-connected graph of order at least $\fref{unavsgtm}(r)$ has the following.
\begin{statement}
 \item\label{unavsubgraph} A subgraph that is an $r$-flower, a cycle of order at least $r$, a chain of $r$ cycles, or a member of the family $\mathscr{K}_{2,r}$.
 \item\label{unavtopminor}
 A topological minor (and hence minor) that is a triangular $r$-flower, $C_r$, a chain of $r$ triangles, or $K_{2,r}$.
\end{statement}
\end{theorem}

We can also obtain a result for multigraphs from \cref{unavsgtm} by subdividing each edge of a multigraph to get a simple graph, then translating the unavoidable substructures into multigraph substructures.
For multigraphs the number of edges is the appropriate measure of largeness, rather than order.
The topological minor and minor orderings are slightly different for multigraphs, because we do not necessarily delete parallel edges created by contractions.

Let $D_r$ be the $r$-edge dipole consisting of $r$ parallel edges between two vertices,
let $S_{2,r}$ be the graph obtained from $K_{1,r}$ by doubling each edge, i.e., replacing each edge with a parallel class of two edges, and let $P_{2,r}$ be the graph obtained from $P_r$ by doubling each edge.
We obtain $D_r$, $S_{2,r}$, and $P_{2,r}$ from $K_{2,r}$, a triangular $r$-flower, and a chain of $r$ triangles, respectively, by contracting vertices of degree 2.

\begin{corollary}\label{cor:2econtm}
For every integer $r\ge 2$, there is an integer $\fref{cor:2econtm}(r)$ such that every $2$-edge-connected multigraph with at least $\fref{cor:2econtm}(r)$ edges has the following.
\begin{statement}
 \item A subgraph that is an $r$-flower, a cycle of order at least $r$, a chain of $r$ cycles, $D_r$, or a member of the family $\mathscr{K}_{2,r}$.
 \item\label{mgtopminorfin}
 A topological minor (and hence minor) that is $S_{2,r}$, $C_r$, $P_{2,r}$, or $D_r$.
\end{statement}
\end{corollary}

We now consider Eulerian subgraphs.  Goddard and LaVey \cite{Goddard2024} proved that for each fixed $t$, there are only finitely many $2$-edge-connected graphs whose maximum closed trail length is at most $t$.  They proved finiteness by showing that there is a bound on the order of such graphs, but did not provide an explicit bound.  This result can be restated as follows.

\begin{theorem}[{Goddard and LaVey \cite[Lemma 4.1]{Goddard2024}}]\label{gl}
For every integer $s \ge 2$, there is an integer $\fref{gl}(s)$ such that every $2$-edge-connected graph of order at least $\fref{gl}(s)$ has an Eulerian subgraph of order at least $s$.
\end{theorem}

Goddard and LaVey \cite[Theorem 4.2]{Goddard2024} used this result to give an upper bound on the number of colors needed to edge-color a large $2$-edge-connected graph so that there is a walk satisfying a certain `rainbow' condition between every pair of vertices. By taking \cref{unavsgtm}\ref{unavsubgraph} with even $r=2s$, we obtain the following strengthening of \cref{gl}.

\begin{corollary}\label{unaveulsubgraph}
For every integer $s \ge 2$ there is an integer $\fref{unaveulsubgraph}(s)$ such that every $2$-edge-connected graph of order at least $\fref{unaveulsubgraph}(s)$ has one of the following Eulerian subgraphs: a $2s$-flower, a cycle of order at least $s$, a chain of $s$ cycles, or a member of the family $\mathscr{K}_{2,2s}$.
\end{corollary}

There is also a multigraph version of \cref{unaveulsubgraph}, which uses number of edges rather than order, and includes $D_{2s}$ as an additional possible Eulerian subgraph.

We also get a result for infinite graphs from \cref{thm:inf2econ}.
Recall that $\mathscr{L}_{\infty}$ consists of subdivisions of the infinite ladder $L_\infty$ where each rung is subdivided at least once.
Let $\mathscr{L}^0_\infty$ be the family of graphs obtained from $L_{\infty}$ by subdividing each of the rail edges an arbitrary number, possibly zero, of times (but without subdividing any rungs).
Note that $L_\infty$ is a member of $\mathscr{L}_\infty^0$.
In the infinite case the minor result is different from the topological minor result, because $L_\infty$ has a triangular $\infty$-flower (and also a one-way infinite chain of triangles) as a minor.

\begin{theorem}\label{infunavsubgraph}
Every infinite $2$-edge-connected graph has the following.
\begin{statement}
 \item A subgraph that is an $\infty$-flower, a one-way infinite chain of cycles, or a member of $\mathscr{K}_{2,\infty} \cup \mathscr{L}_\infty^0 \cup \mathscr{L}_\infty$. 
 \item\label{topminorinf}
 A topological minor that is a triangular $\infty$-flower, a one-way infinite chain of triangles, $L_\infty$, or $K_{2,\infty}$.
 \item A minor that is a triangular $\infty$-flower, a one-way infinite chain of triangles, or $K_{2,\infty}$.
\end{statement}
\end{theorem}

If we consider multigraphs instead of simple graphs, we get the following result by a similar process to the finite case.  Note that a multigraph is infinite if it has an infinite number of edges or vertices.
The multigraph $P_{2,\infty}$ is obtained from $P_{\infty}$ by doubling each edge, $S_{2,\infty}$ is obtained from $K_{1,\infty}$ by doubling each edge, and $D_\infty$ denotes a dipole with a countably infinite number of edges.

\begin{corollary}\label{cor:2ectminf}
Every infinite $2$-edge-connected multigraph has the following.
\begin{statement}
 \item A subgraph that is an $\infty$-flower, a one-way infinite chain of cycles, $D_\infty$, or a member of $\mathscr{K}_{2,\infty} \cup \mathscr{L}_{\infty}^0\cup \mathscr{L}_\infty$.
 \item\label{mgtopminorinf}
 A topological minor that is $S_{2,\infty}$, $P_{2,\infty}$, $L_\infty$, or $D_\infty$.
 \item A minor that is $S_{2,\infty}$, $P_{2,\infty}$, or $D_\infty$.
\end{statement}
\end{corollary}

\crefWithTheorem{mgtopminorinf} also appears in the PhD thesis of Qualls \cite[Theorem 1.2.7]{brittthesis}.

An infinite graph is \emph{Eulerian} if it has a two-way infinite trail that uses every edge.  We can guarantee the existence of an infinite Eulerian subgraph with restricted structure. 

\begin{corollary}
Every infinite $2$-edge-connected graph has an Eulerian subgraph that is a two-way infinite path, an $\infty$-flower, an element of $\mathscr{K}_{2,\infty}$, or a one-way infinite chain of cycles.
\end{corollary}

For multigraphs we must add $D_\infty$ as an additional possibility.  Note that a two-way infinite path is not $2$-edge-connected, but it is the infinite analog of a cycle because it is $2$-regular and connected.  We must include the two-way infinite path because it is the only infinite Eulerian subgraph of some $2$-edge-connected infinite ladders (including all members of $\mathscr{L}^0_\infty$).

\subsection{Unavoidable induced topological minors and induced minors}\label{ss:unavoidindminor}

In this subsection we consider two less common orderings, the induced topological minor and induced minor orderings (for simple graphs only, not for multigraphs).  We say that $H$ is an \emph{induced minor} of $G$ if a graph isomorphic to $H$ can be obtained from $G$ by vertex deletions and edge contractions.  If, moreover, each contracted edge is incident with a vertex of degree $2$ then we say $H$ is an \emph{induced topological minor} of $G$, which is equivalent to an induced subgraph of $G$ being isomorphic to a subdivision of $H$.  The orderings induced subgraph, induced topological minor, induced minor, and minor form a ranked sequence of orderings, from strongest to weakest.  Also, the induced topological minor ordering is stronger than the topological minor ordering, giving a third ranked sequence of orderings, namely induced subgraph, induced topological minor, topological minor, minor, again from strongest to weakest.

A super-clean pinched ladder is \emph{(topologically) irreducible} if it is a triangle or the only vertices of degree $2$ are the initial and final vertices.  Super-clean pinched ladders in a chain can be made irreducible by contracting edges incident with vertices of degree $2$, and can be reduced to triangles by contracting arbitrary edges.  However, we cannot make single large super-clean pinched ladders irreducible, as that may decrease their order in an uncontrolled way.  But if we are allowed to contract arbitrary edges, we can contract both edges in the span of each cross, which reduces the number of vertices by a factor of at most $\frac12$, and does not introduce any nontrivial embedded fans.  Therefore,  \cref{thm:2econ} yields the following.

\begin{theorem}\label{thm:indminorfin}
For every integer $r \ge 3$, there is an integer $\fref{thm:indminorfin}(r)$ such that every $2$-edge-connected graph of order at least $\fref{thm:indminorfin}(r)$ has the following.
\begin{statement}
 \item An induced topological minor that is $K_r$, a triangular $r$-flower, a super-clean pinched ladder of order at least $r$, a chain of $r$ irreducible super-clean pinched ladders, or $K_{1,1,r}$.
 \item An induced minor that is $K_r$, a triangular $r$-flower, a super-clean pinched ladder with no crosses of order at least $r$, a chain of $r$ triangles, or $K_{1,1,r}$.
\end{statement}\end{theorem}

In the infinite case we say a one-way infinite super-clean pinched ladder is \emph{(topologically) irreducible} if the only vertex of degree $2$ is the initial vertex.  An infinite super-clean pinched ladder has infinitely many rungs, so we can make it irreducible by contracting rail edges incident with vertices of degree $2$ without affecting the fact that we have an infinite graph.
Define $L_\infty\tri$ and $L_\infty\tritri$ to be the graphs obtained by replacing all vertices on one or both rails of $L_\infty$, respectively, by triangles.  Then all members of $\Loneinf \cup \Ltwoinf \cup \Lthreeinf$ are subdivisions of $L_\infty$, $L_\infty\tri$, or $L_\infty\tritri$.
Moreover, each of $L_\infty$, $L_\infty\tri$, $L_\infty\tritri$, or an irreducible one-way infinite super-clean pinched ladder has a triangular $\infty$-flower as an induced minor.
Thus, \cref{thm:inf2econ} gives the following.

\begin{theorem}\label{thm:indminorinf}
Every infinite $2$-edge-connected graph has the following.
\begin{statement}
 \item An induced topological minor that is $K_\infty$, a triangular $\infty$-flower, a one-way infinite chain of irreducible super-clean pinched ladders, an irreducible one-way infinite super-clean pinched ladder, $K_{1,1,\infty}$, $L_\infty$, $L_\infty\tri$, or $L_\infty\tritri$.
 \item An induced minor that is $K_\infty$, a triangular $\infty$-flower, a one-way infinite chain of triangles, or $K_{1,1,\infty}$.
\end{statement}\end{theorem}

\section*{Acknowledgments}

The second author is grateful for support from the Simons Foundation under award  MPS-TSM-00002760.

We thank Guoli Ding and Brittian Qualls for helpful discussions.  We also thank two anonymous referees of an earlier version of this paper whose comments helped to improve the presentation.  In particular, we express our great appreciation to one referee who pointed out a way to significantly simplify the proofs of \cref{thm:2econ,thm:inf2econ}.

\bibliographystyle{hplain}
\bibliography{bibliography}{}
\vspace{-5mm}

\end{document}